\documentclass[12pt,a4paper]{article}
\usepackage{amsfonts,latexsym,amsmath, amssymb, amscd, amsthm,stmaryrd}
\usepackage{mathrsfs,dsfont}
\usepackage{subfigure} 
\usepackage{bm,bbm}
\usepackage{appendix}
\usepackage{color}
\usepackage{geometry}
\usepackage{cite}
\usepackage{cases}
\usepackage[utf8]{inputenc}
\usepackage{graphicx,float,enumerate}
\usepackage[T1]{fontenc}
\usepackage{authblk}

\numberwithin{equation}{section}
\usepackage[colorlinks=true,citecolor=red,linkcolor=blue,urlcolor=RubineRed,pdfpagetransition=Blinds,pdftoolbar=false,pdfmenubar=false]{hyperref}

\def\varep{\varepsilon}

\def\N{\mathbb{N}}

\def\R{\mathbb{R}}

\newcommand{\be}{\begin{equation}}
\newcommand{\ee}{\end{equation}}
\newcommand{\baa}{\begin{array}}
\newcommand{\eaa}{\end{array}}

\newtheorem{theorem}{Theorem}[section]
\newtheorem{lemma}[theorem]{Lemma}
\newtheorem{proposition}[theorem]{Proposition}

\newtheorem{definition}[theorem]{Definition}

\allowdisplaybreaks

\title{\bf{KPP transition fronts in a one-dimensional two-patch habitat}}
\author{Fran\c cois Hamel\thanks{Aix-Marseille Univ, CNRS, I2M, Marseille, France (\texttt{francois.hamel@univ-amu.fr}). This work has received funding from Excellence Initiative of Aix-Marseille Universit\'e~-~A*MIDEX, a French ``Investissements d'Avenir'' programme, and from the French ANR RESISTE (ANR-18-CE45-0019) and ReaCh (ANR-23-CE40-0023-02) projects.} \ \& Mingmin Zhang\thanks{CNRS, UMR 5219, Institut de Math\'ematiques de Toulouse, Universit\'{e} de Toulouse, UPS IMT, F-31062 Toulouse Cedex 9, France (\texttt{mingmin.zhang.math@gmail.com}). M.Z. acknowledges support from the ANR via the project Indyana under grant agreement ANR-21-CE40-0008, and the project ReaCh under grant agreement ANR-23-CE40-0023-02.}}

\date{}
\begin{document}
	
\maketitle

\centerline{\it To Professor James D. Murray}
\centerline{\it in admiration and recognition of his great achievements in mathematical biology}
\vskip 0.5cm

\begin{abstract} 	
\noindent{}This paper is concerned with the existence of transition fronts for a one-dimensional two-patch model with KPP reaction terms. Density and flux conditions are imposed at the interface between the two patches. We first construct a pair of suitable super- and subsolutions by making full use of information of the leading edges of two KPP fronts and gluing them through the interface conditions. Then, an entire solution obtained thanks to a limiting argument is shown to be a transition front moving from one patch to the other one. This propagating solution admits asymptotic past and future speeds, and it connects two different fronts, each associated with one of the two patches. The paper thus provides the first example of a transition front for a KPP-type two-patch model with interface conditions.
\vskip 0.1cm
\noindent{\small{\it Mathematics Subject Classification}: 35B08; 35K57.}
\vskip 0.1cm
\noindent{\small{\it Key words}: Reaction-diffusion equations; Transition fronts; KPP reactions; Interface conditions.}
\end{abstract}


\section{Introduction}


\subsection{The model}
In this paper, we deal with the existence of transition fronts of the following two-patch problem with interface conditions:
\begin{align}
\label{model}
\begin{cases}
u_t =d_1  u_{xx}+f_1(u),~&t\in\mathbb{R},\ x<0,\cr
u_t =d_2 u_{xx}+f_2(u),~&t\in\mathbb{R},\ x>0,\cr
u(t,0^-)=u(t,0^+),~&t\in\mathbb{R},\cr
u_x(t,0^-)=\sigma u_x(t,0^+), ~&t\in\mathbb{R},
\end{cases}
\end{align}
in which $\sigma>0$, $d_i>0$ (for $i=1,2$), and the functions $f_i\in C^1(\R)$ (for $i=1,2$) are of Fisher-KPP type:
\be\label{hypkpp}
f_i(0)\!=\!f_i(K_i)\!=\!0,\ 0\!<\!f_i(s)\!\le\!f'_i(0)s\ \hbox{for all}\ s\in(0,K_i),\ f_i'(K_i)<0,\ f_i\!<\!0\ \hbox{in}\ (K_i,+\infty),
\ee
for some $K_i>0$.

From the perspective of ecological dynamics of invasive species, when spreading across a landscape, species encounter different habitat types, and their movement behavior as well as population dynamics may change according to landscape type.  Here, we consider the evolution of a population density under the effect of diffusion and growth, in a medium made up of two different semi-infinite one-dimensional habitats separated by an interface, under a simple assumption that each patch is homogeneous but the two patches may differ so that the diffusion coefficients and the reaction terms (i.e.~net population growth rates) may differ.

Originally, such kind of patchy model with novel interface matching conditions is from a recent work of Maciel and Lutscher \cite{ML2013} which itself is based on the work of Ovaskainen and Cornell \cite{OC2003}. There,  the population density $\widetilde{u}=\widetilde{u}(t,x)$ in such a two-patch landscape satisfies
 \begin{align*}
 	\begin{cases}
 		\widetilde u_t = d_1 \widetilde u_{xx} + \widetilde f_1(\widetilde u),~&t\in\mathbb{R},\ x<0,\cr
 		\widetilde u_t = d_2 \widetilde u_{xx} + \widetilde f_2(\widetilde u),~&t\in\mathbb{R},\ x>0,\cr
 		(1-\alpha)d_1\widetilde u(t,0^-)=\alpha d_2\widetilde u(t,0^+),~&t\in\mathbb{R},\cr
 		d_1\widetilde u_x(t,0^-)=d_2 \widetilde u_x(t,0^+), ~&t\in\mathbb{R},
 	\end{cases}
 \end{align*}
where it is assumed that individuals at the interface may show a preference for one or the other patch type measured by the parameter $\alpha\in(0,1)$ ($\alpha>1/2$ indicates a preference for the left patch $(-\infty,0)$ and $\alpha<1/2$ for the right patch $(0,+\infty)$), and the interface is assumed to be neutral with respect to reaction dynamics (i.e.~no individuals are born or die from crossing the interface) so that the flux is continuous at the interface (such continuity property of the flux implies mass conservation in the absence of reaction terms). It is observed that the population density may be discontinuous at the interface as long as $k:=\frac{\alpha}{1-\alpha}\,\frac{d_2}{d_1}$ is not equal to~1. We refer to \cite{ML2013} for a detailed derivation of this condition from a random walk and a thorough discussion of the biological implications. This type of model has been used to study questions of persistence and spread~\cite{AL2019,ML2015} and evolutionary stable movement strategies~\cite{MCCL2020}, whereas related flux matching conditions between adjacent higher-dimensional domains have been considered in~\cite{BRT2021}. As far as the reaction-diffusion equations in each of the two patches $\{x<0\}$ and $\{x>0\}$ are concerned, they are standard equations used to describe biological invasions in mathematical biology and ecology, see e.g. Murray's book~\cite{Murray}, as well as~\cite{CC2003,R2013,SK1997,Z2003}.
 
From mathematical viewpoint, the discontinuity of the density at $x=0$ makes the problem quite delicate to study, and  it turns out to be much easier to rescale the model (by setting $u(t,x)=\widetilde u(t,x)$ in patch~1, $u(t,x)=k\widetilde u(t,x)$ in patch~2 with $k=\frac{\alpha}{1-\alpha}\,\frac{d_2}{d_1}$, $f_1=\widetilde{f}_1$ and $f_2(s)=k\widetilde f_2(s/k)$) in such a way that the matching conditions become continuous in the density, and this is exactly where the equivalent problem \eqref{model} - our objective - comes from, with $\sigma=(1-\alpha)/\alpha>0$. Our present  work, concerning the existence of transition fronts for problem~\eqref{model}, is a continuation of the rigorous analysis towards a better understanding of propagation phenomena in such models \cite{HLZ, HLZ2, SKW2015}.


\subsection{Traveling fronts and transition fronts for homogeneous or more general equations}

The issue of traveling fronts for the classical Fisher-KPP equation 
\begin{equation}
\label{KPP-1}
u_t=u_{xx}+f(u),~~t>0,~x\in\mathbb{R},
\end{equation}
has been addressed in the pioneering works of Fisher \cite{F1937} and Kolmogorov, Petrovskii and Piskunov \cite{KPP1937}, where $f$ is of Fisher-KPP type: 
$$f(0)=f(1)=0\ \hbox{ and }\ 0<f(s)\le f'(0)s\ \hbox{ for all }s\in (0,1).$$
It was proved that \eqref{KPP-1} admits traveling front solutions $u(t,x)=\varphi(x-ct)$ with $\varphi:\R\to(0,1)$ and $\varphi(-\infty)=1$, $\varphi(+\infty)=0$, if and only if $c\ge c^*:=2\sqrt{f'(0)}$, where  $c\in\R$ is the front speed and $\varphi=\varphi_c$ is the front profile (depending on $c$). It is also known \cite{KPP1937} that the front with minimal speed $c^*$ attracts the solutions of the Cauchy problem starting from the Heaviside function $\mathbbm{1}_{(-\infty,0)}$ in a certain sense, see e.g.~\cite{B1983,HNRR2013,L1985,U1978}. Among many other references, the existence of traveling fronts to more general types of reaction terms was discussed in~\cite{AW1,AW2,F1979,FM1977,GK2004,VVV1994}.

Afterwards, heterogeneity has been taken into account in the investigation of propagating solutions of non-homogeneous reaction-diffusion equations, for which standard traveling fronts do not exist in general. Especially, when the equation is spatially or temporally periodic, the notion of pulsating traveling fronts has been developed in one-dimensional or higher-dimensional domains, see e.g.~\cite{BH2002,SKT1986,X2000}. Still analogously to the homogeneous case, pulsating traveling fronts with speed $c$, for KPP-type periodic equations exist, if and only if $c\ge c^*$, where the minimal wave speed $c^*$ has a variational expression in terms of periodic principal eigenvalues of some linear operators \cite{BHN2005,BHR2005,DD,LYZ2006,LZ2007,LZ2010,N2009, NRX2005,W2002}.

Later on, the study of fronts for reaction-diffusion equations in more general heterogeneous media has been given considerable attention. A generalization of the notion of traveling fronts in such media has been given in~\cite{BH2007,BH2012} in general domains in any space dimension, see also~\cite{M2003,S2004} for related definitions in particular cases. When applied to a one-dimensional equation such as~\eqref{KPP-1}, the definition is as follows: a generalized transition front of~\eqref{KPP-1} connecting $1$ and $0$ is a time-global solution $u:\R\times\R\to(0,1)$ for which there exists a locally bounded function $X: \R\to\R$ such that 
$$\lim_{x\to -\infty} u(t,x+X(t))=1,~~~\lim_{x\to +\infty} u(t,x+X(t))=0,~~\text{uniformly in}~t\in\R,$$
where $X(t)\in\R$ reflects the position of the transition front as time progresses. Moreover, such a transition front $u$ has a global mean speed $w\in\R$ if 
$$\lim_{|t-s|\to+\infty}\frac{X(t)-X(s)}{t-s}=w.$$
This definition covers all the classical examples of travelling and pulsating fronts. There has been a large literature devoted to transition fronts for reaction-diffusion equations of the type~\eqref{KPP-1} with homogeneous or heterogeneous KPP-type reactions $f$ in one dimension, see e.g.~\cite{BH2012,NR2009,Shen2017,TZZ2014,Z2012}, and in higher dimensions, see e.g.~\cite{AHLTZ2019,BH2012,Z2012,Z2017}. Whereas transition fronts exist in general for ignition-type equations \cite{MRS2010, NR2009,Z2013,Z2017b}, transition fronts for spatially heterogeneous KPP equations do not exist in general~\cite{HN2022,NRRZ}. The existence of transition fronts for KPP time-dependent equations has been proved when the coefficients are assumed to be uniquely ergodic in~\cite{Shen2011-} and in a general framework~\cite{NR2012}. Existence results have been further extended to KPP equations with time-heterogeneous reaction terms and space-periodic diffusion and advection terms in~\cite{RR2014}, and to general time-heterogeneous and space-periodic equations in~\cite{NR2015}. The existence of generalized transition fronts for KPP equations in one-dimensional almost periodic media was investigated in~\cite{NR2017}, and for monostable equations in time recurrent and spatially periodic media in~\cite{Shen2011+}. On the other hand, the existence and asymptotic dynamics of transition fronts in time-dependent KPP type equations was analyzed in~\cite{HR-SIMA}, where the media are specifically asymptotically homogeneous as $t\to\pm \infty$ with two possibly different limits. Transition fronts for homogeneous KPP equation~\eqref{KPP-1} as well as the set of their admissible asymptotic past and future speeds and their asymptotic profiles were studied in~\cite{HN2001,HR-TAMS}, it was proved in particular in~\cite{HR-TAMS} that the transition fronts of \eqref{KPP-1} can only accelerate. The existence of critical transition waves, which are by definition steeper than any other solution, was addressed in~\cite{N2015} for general spatially heterogeneous one-dimensional equations.
 
In contrast, the one-dimensional two-patch model~\eqref{model} we are considering here is spatially heterogeneous in a simple fashion but very different from existing ones, in the sense that each patch is homogeneous and the two patches match each other through particular interface conditions at $x=0$. It is well known that traveling fronts for homogeneous KPP equations of the type~\eqref{KPP-1} are pulled by their tails~\cite{GGHR,S1976} and the spreading speed of solutions of the associated Cauchy problem converging to $0$ as $x\to+\infty$ is determined by the exponential decay of the initial condition~\cite{B1983,L1985,U1978}. Therefore, in order to show the existence of propagating solutions of~\eqref{model} that decay to $0$ as $x\to+\infty$, a natural attempt is to make full use of information of the leading edges of the KPP fronts in both patches, and to match them through the interface conditions. This is exactly the idea we will carry out in the paper. More precisely, we find out a suitable pair of super- and subsolutions, which leads by a constructive limiting argument to the existence of rightwards propagating transition fronts of~\eqref{model} with explicit asymptotic past and future speeds (see Definition~\ref{def} below for these notions of speeds). In a sense, whereas standard traveling fronts can not exist in general due to the interface conditions and the different diffusion and reaction terms in the two patches, model~\eqref{model} with its interface conditions is robust enough to allow the existence of non-trivial propagating solutions in the form of transition fronts connecting two different steady states. Up to the best of our knowledge, at the exception of a recent work~\cite{JM2019} on propagation or blocking phenomena for a related system made up of copies of a bistable equation in multiple disjoint half-lines with a junction, the topic of transition fronts for a patch model like~\eqref{model} is quite new and there had been no existing results on it up to now. However, it is unclear at this stage whether other kinds of transition fronts exist or not, and the question of the classification of such transition fronts, which is actually still open even in the homogeneous case~\eqref{KPP-1}, goes much beyond the scope of this article, and we leave it open for a future work.


\section{The main result}

Before stating our main result, we make precise the notion of classical solution of~\eqref{model} and we recall some fundamental results of~\cite{HLZ2,HLZ} on the Cauchy problem and the comparison principle associated with~\eqref{model}.


\subsection{What is known about~\eqref{model}}

Throughout this paper, we set 
$$I_1:=(-\infty,0) ~~\text{and}~~I_2:=(0,+\infty).$$ 
In the sequel, by $C^{1;2}_{t;x}$, we understand the class of functions which are of class $C^1$ in $t$ and $C^2$ in~$x$. Similarly, for $\gamma>0$, $C^{1,\gamma;2,\gamma}_{t;x}$ is the class of functions which are $C^{1,\gamma}$ in $t$ and $C^{2,\gamma}$ in $x$. By a solution to the Cauchy problem~\eqref{model} associated with a continuous bounded initial datum~$u_0$, we mean a classical solution in the following sense.

\begin{definition}[\!\!\cite{HLZ2}]
\label{def1}
For $T\in(0,+\infty]$, we say that a continuous function $u:[0,T)\times\R\to\R$ is a classical solution of the Cauchy problem~\eqref{model} in $[0,T)\times\R$ with initial datum~$u_0$, if $u(0,\cdot)=u_0$ in~$\R$, if~$u|_{(0,T)\times\overline{I_i}}\in C^{1;2}_{t;x}\big((0,T)\times \overline{I_i}\big)$ $($for $i=1,2$$)$, and if all identities in~\eqref{model} are satisfied pointwise for $0<t<T$.
\end{definition}

We recall the well-posedness of the Cauchy problem \eqref{model} as well as regularity estimates of the solution.

\begin{proposition}[\!\!\cite{HLZ,HLZ2}]
\label{prop-wellposedness+Schauder}
For any nonnegative bounded continuous function $u_0:\R\to\R$ and for any $\gamma\in(0,1/2)$, there is a unique nonnegative bounded classical solution~$u$ of~\eqref{model} in~$[0,+\infty)\times\R$ with initial datum $u_0$ such that, for any $\tau>0$ and $A>0$, 
\begin{align*}
\Vert u|_{[\tau,+\infty)\times[-A,0]}\Vert_{C^{1,\gamma;2,\gamma}_{t;x}([\tau,+\infty)\times[-A,0])}+\Vert u|_{[\tau,+\infty)\times[0,A]}\Vert_{C^{1,\gamma;2,\gamma}_{t;x}([\tau,+\infty)\times[0,A])}\le C,
\end{align*}
with a positive constant $C$ depending on $\tau$, $A$, $\gamma$, $d_{1,2}$, $f_{1,2}$, $\sigma$, and $\Vert u_0 \Vert_{L^\infty(\mathbb{R})}$. Moreover, 
$$0\le u(t,x)\le\max(K_1,K_2,\|u_0\|_{L^\infty(\R)})\ \hbox{ for all $(t,x)\in[0,+\infty)\times\R$},$$
and $u(t,x)>0$ for all $(t,x)\in(0,+\infty)\times\R$ if $u_0\not\equiv0$ in $\R$. Lastly, the solutions depend monotonically and continuously on the initial data, in the sense that if $u_0\le v_0$ then the corres\-ponding solutions satisfy $u\le v$ in~$[0,+\infty)\times\R$, and for any $T\in(0,+\infty)$ the map~$u_0\mapsto u$ is continuous from $C^+(\R)\cap L^\infty(\R)$ to~$C([0,T]\times\R)\cap L^\infty([0,T]\times\R)$ equipped with the sup norms, where~$C^+(\R)$ denotes the set of nonnegative continuous functions in $\R$.
\end{proposition}

The existence in Proposition~\ref{prop-wellposedness+Schauder} can be proved by following the proof of~\cite[Theorem~2.2]{HLZ}, namely by solving approximated problems in bounded intervals $[-n,n]$, using a priori estimates, and passing to the limit as $n\to+\infty$.

We also recall the definition of super- and subsolutions for~\eqref{model} and the comparison principle in the following two statements.

\begin{definition}[\!\!\cite{HLZ2}]
\label{def2}
For $T\in(0,+\infty)$, a bounded continuous function $\overline{u}:[0,T]\times\R\to\mathbb{R}$ is called a supersolution of~\eqref{model} in~$[0,T]\times\R$, if~$\overline{u}|_{(0,T]\times\overline{I_i}}\in  C^{1;2}_{t;x}((0,T]\times\overline{I_i})$ $($for~$i=1,2$$)$, if~$\overline{u}_t(t,x)\ge d_i\overline{u}_{xx}(t,x)+f_i(\overline{u}(t,x))$ for $i=1,2$ and for all $0<t\le T$ and $x\in I_i$, and~if
$$\overline{u}_x(t,0^-)\ge \sigma \overline{u}_x(t,0^+)~~\text{for all}~t\in(0,T].$$
Subsolutions are defined in a similar way with all the inequality signs above reversed.
\end{definition}

\begin{proposition}[\!\!\cite{HLZ}]
\label{prop-cp}
For $T\in(0,+\infty)$, let $\overline{u}$ and $\underline{u}$ be, respectively, a super- and a subsolution of~\eqref{model} in~$[0,T]\times\R$, and assume that $\overline{u}(0,\cdot)\ge\underline{u}(0,\cdot)$ in~$\R$. Then,~$\overline u\ge \underline u$ in~$[0,T]\times\R$ and, if $\overline{u}(0,\cdot)\not\equiv\underline{u}(0,\cdot)$ in $\R$, then $\overline{u}>\underline{u}$ in $(0,T]\times\R$.
\end{proposition}

Proposition~\ref{prop-cp} is derived from~\cite[Proposition~A.3]{HLZ}. From the proof of~\cite[Proposition~A.3]{HLZ}, the above result partly extends to the case where $\overline{u}$ and $\underline{u}$ are generalized super- and sub-solutions, that is,
$$\overline{u}(t,x)=\min(\overline{u}_{1,i}(t,x),\ldots,\overline{u}_{p_i,i}(t,x))\ \hbox{ and }\ \underline{u}(t,x)=\max(\underline{u}_{1,i}(t,x),\ldots,\underline{u}_{q_i,i}(t,x))$$
for $i=1,2$, $t\in[0,T]$ and $x\in I_i$, with any positive integers $p_i$ and $q_i$. Here, the functions $\overline{u}_{j,i}$ and $\underline{u}_{k,i}$, for $i=1,2$, $j\in\llbracket1,p_i\rrbracket$ and $k\in\llbracket1,q_i\rrbracket$, are all assumed to be defined, bounded and continuous in $[0,T]\times\overline{I_i}$, and of class $C^{1;2}_{t;x}((0,T]\times\overline{I_i})$. One also assumes that:
\begin{itemize}
\item $(\overline{u}_{j,i})_t(t,x)\ge d_i(\overline{u}_{j,i})_{xx}(t,x)+f_i(\overline{u}_{j,i}(t,x))$ for $i=1,2$, $j\in\llbracket1,p_i\rrbracket$ and $(t,x)\in(0,T]\times I_i$ such that $\overline{u}(t,x)=\overline{u}_{j,i}(t,x)$;
\item $(\underline{u}_{k,i})_t(t,x)\le d_i(\underline{u}_{k,i})_{xx}(t,x)+f_i(\underline{u}_{k,i}(t,x))$ for $i=1,2$, $k\in\llbracket1,q_i\rrbracket$ and $(t,x)\in(0,T]\times I_i$ such that $\underline{u}(t,x)=\underline{u}_{k,i}(t,x)$;
\item there are $j_i\in\llbracket1,p_i\rrbracket$ (for $i=1,2$) and $r>0$ such that $\overline{u}_{j_i,i}(t,x)=\overline{u}(t,x)$ for all $t\in[0,T]$ and $x\in I_i\cap(-r,r)$, $\overline{u}_{j_1,1}(t,0^-)=\overline{u}_{j_2,2}(t,0^+)$ for all $t\in[0,T]$, and $(\overline{u}_{j_1,1})_x(t,0^-)\ge\sigma(\overline{u}_{j_2,2})_x(t,0^+)$ for all $t\in(0,T]$; 
\item there are $k_i\in\llbracket1,q_i\rrbracket$ (for $i=1,2$) and $s>0$ such that $\underline{u}_{k_i,i}(t,x)=\underline{u}(t,x)$ for all $t\in[0,T]$ and $x\in I_i\cap(-s,s)$, $\underline{u}_{k_1,1}(t,0^-)=\underline{u}_{k_2,2}(t,0^+)$ for all $t\in[0,T]$, and $(\underline{u}_{k_1,1})_x(t,0^-)\le\sigma(\underline{u}_{k_2,2})_x(t,0^+)$ for all $t\in(0,T]$.
\end{itemize}
From the above assumptions, the functions $\overline{u}$ and $\underline{u}$ can be extended continuously in $[0,T]\times\R$ (that is, including at the interface $x=0$). The extension of Proposition~\ref{prop-cp} asserts that, if~$\overline{u}(0,\cdot)\ge\underline{u}(0,\cdot)$ in~$\R$, then $\overline u\ge \underline u$ in~$[0,T]\times\R$.

Lastly, by a classical stationary solution of~\eqref{model}, we mean a continuous function $U: \mathbb{R}\to\mathbb{R}$ such that $U|_{\overline{I_i}}\in C^2(\overline{I_i})$ (for $i=1,2$) and all identities in~\eqref{model} are satisfied pointwise, but without any dependence on $t$. It is known from \cite{HLZ2} that~\eqref{model} admits a unique positive bounded stationary solution $V$. Moreover,
$$V(-\infty)=K_1\ \hbox{ and }\ V(+\infty)=K_2,$$
and $V$ is strictly monotone if $K_1\neq K_2$, whereas $V$ is constant if $K_1=K_2$.


\subsection{Some notations and the notion of transition front connecting $V$ and $0$ for~\eqref{model}}

We recall that, for $i=1,2$, the homogeneous Fisher-KPP equation
$$u_t=d_i u_{xx}+f_i(u),~~t\in\mathbb{R},~x\in\mathbb{R},$$
admits standard traveling fronts $\phi_i(x-c_it)$ such that
\be\label{TW}
d_i\phi_i''+c_i\phi_i'+f_i(\phi_i)=0~\text{in}~\mathbb{R},~~\phi_i'<0~\text{in}~\mathbb{R},~~\phi_i(-\infty)=K_i,~~\phi_i(+\infty)=0,
\ee
if and only if $c_i\ge c^*_i:=2\sqrt{d_i\mu_i}$, where we denote
\be\label{defmui}
\mu_i:=f_i'(0)>0
\ee
for convenience (the functions $\phi_i$ also depend on the speeds $c_i$ and we should therefore write~$\phi_{i,c_i}$, but we kept the notation $\phi_i$ for the sake of simplicity, as the considered speeds $c_1$ and $c_2$ in the main result will be explicit). Furthermore, the functions $\phi_i$ are  unique up to shifts. By~\cite{AW2}, throughout this paper, we assume without loss of generality, up to shifts, that the traveling wave profiles $\phi_i$ (for $i=1,2$) satisfy the following normalization conditions:
\begin{equation}
\label{TW-decay}
\begin{aligned}
\begin{cases}
\phi_i(\xi)\sim e^{-\lambda_i\xi}~&\text{for}~c_i>c^*_i\cr
\phi_i^*(\xi)\sim \xi e^{-\lambda_i^*\xi}~&\text{for}~c_i=c^*_i
\end{cases}
~~\text{as}~\xi\to+\infty,
\end{aligned}
\end{equation}
where 
$$\lambda_i:=\frac{c_i-\sqrt{c_i^2-4d_i\mu_i}}{2d_i}~~\text{for}~c_i> c^*_i\ \hbox{ and }\ \lambda_i^*:=\frac{c_i^*}{2d_i}=\sqrt{\frac{\mu_i}{d_i}}~~\text{for}~c_i= c^*_i.$$
With the normalization \eqref{TW-decay}, it is also known that
\be\label{ineqphii}
0<\phi_i(\xi)\le e^{-\lambda_i\xi}\ \hbox{ for all $c_i>c^*_i$ and $\xi\in\mathbb{R}$}.
\ee

Throughout this paper, we will further assume that the functions $f_i$ (for $i=1,2$) satisfy the following regularity property:
\begin{equation}
\label{f-w}
f_i(s)\ge \mu_is-C s^{1+\omega}~~\text{for all}~s\in[0,K_i],
\end{equation}
for some $C>0$ and $\omega>0$. 

\begin{definition}
\label{def}
For problem \eqref{model}, a transition front connecting the unique positive bounded stationary solution~$V$ and~$0$ is a time-global classical solution $u$ for which there exists a locally bounded function $X: \mathbb{R}\to\mathbb{R}$ such that
\begin{align}
\label{X(t)}
\begin{cases}
u(t,x)- V(x)\to 0~~&\text{as}~x-X(t)\to-\infty\cr
u(t,x)\to 0~~&\text{as}~x-X(t)\to+\infty
\end{cases}
~~~\text{uniformly in}~t\in\mathbb{R}.
\end{align}
Moreover, we say that a transition front connecting  $V$ and $0$ for problem~\eqref{model} has an asymptotic past speed $c_-\in\mathbb{R}$ $($resp. an asymptotic future speed $c_+\in\mathbb{R}$$)$, if
$$\frac{X(t)}{t}\to c_-~\text{as}~t\to -\infty~~\Big(\text{resp.}~\frac{X(t)}{t}\to c_+~\text{as}~t\to +\infty\Big).$$
\end{definition}

Observe that any transition front $u$ of~\eqref{model} connecting $V$ and $0$ necessarily satisfies
\begin{align}\label{uK1}
\begin{cases}
u(t,x)\to K_1~~&\text{as}~x\to-\infty\cr
u(t,x)\to 0~~&\text{as}~x\to+\infty
\end{cases}
~~~\text{locally uniformly in}~t\in\mathbb{R}.
\end{align}
Furthermore, if $X(t)\to+\infty$ as $t\to+\infty$, then $u(t,\cdot)\to V$ as $t\to+\infty$ locally uniformly in $\mathbb{R}$, and even uniformly in each interval $(-\infty,A]$ with $A\in\R$.


\subsection{The main result}

The main result of this paper is the following theorem on the existence of transition fronts connecting $V$ and $0$ for problem~\eqref{model}. From a biological point of view, the transition front constructed below can be interpreted as an alien species invading the left patch from $-\infty$ with asymptotic (past) speed $c_1$ and propagating across the interface and spreading in the right patch with asymptotic (future) speed $c_2$.

\begin{theorem}
\label{thm-1}
Assume that $d_2\neq d_1\sigma^2$ and that $\mu_1,\mu_2$ defined in~\eqref{defmui} satisfy
\begin{equation}
\label{mu1-mu2}
\begin{aligned}
\begin{cases}
\displaystyle\mu_1<\mu_2< \mu_1\Big(2-\frac{d_2}{d_1\sigma^2}\Big)~~\hbox{if}~d_2<d_1\sigma^2,\\
\displaystyle\mu_2<\mu_1< \mu_2\Big(2-\frac{d_1\sigma^2}{d_2}\Big) ~~\hbox{if}~d_2>d_1\sigma^2.
\end{cases}
\end{aligned}
\end{equation}
Define
\begin{equation}
\label{lambda1-2}
\lambda_2:=\sqrt{\frac{\mu_2-\mu_1}{d_1\sigma^2-d_2}}>0\ \hbox{ and }\ \lambda_1:=\sigma \lambda_2>0.
\end{equation}
Then, there exist $c_2\in(c^*_2,+\infty)$ and $c_1\in(c^*_1,+\infty)$ given by 
\begin{equation}
\label{c-12}
c_2:=d_2\lambda_2+\frac{\mu_2}{\lambda_2}\ \hbox{ and }\ c_1:=\frac{\lambda_2c_2}{\lambda_1}=\frac{c_2}{\sigma},
\end{equation}
such that \eqref{model} admits a transition front connecting the unique positive bounded stationary solution~$V$ and~$0$, with asymptotic past speed $c_1$ and asymptotic future speed $c_2$, in the sense of Definition~$\ref{def}$. Furthermore,
\be\label{lim+infty0}
\lim_{t\to-\infty}\Big(\sup_{x\in\R}|u(t,x)-\phi_1(x-c_1t)|\Big)=0,
\ee
and
\be\label{lim+infty}
\lim_{(t,x)\to(+\infty,+\infty)}|u(t,x)-\phi_2(x-c_2t)|=0,
\ee
where $\phi_i$ $($for $i=1,2)$ are given by~\eqref{TW} and~\eqref{TW-decay}, with the speeds $c_i$ defined in~\eqref{c-12}
\end{theorem}

The limit in~\eqref{lim+infty} means that, for every $\varepsilon>0$, there is $A_\varepsilon>0$ such that
$$|u(t,x)-\phi_2(x-c_2t)|\le\varepsilon\ \hbox{ for all $t\ge A_\varepsilon$ and $x\ge A_\varepsilon$}.$$
In other words, the solution $u$ looks like the front $\phi_2(x-c_2t)$ at large time $t$ and for large~$x$. We point out that this convergence can not be uniform with respect to $x$ in $\R$ as soon as $K_1\neq K_2$, since $u(t,-\infty)=K_1$ for each $t\in\R$ by~\eqref{uK1}, whereas $\phi_2(-\infty)=K_2$. 

Let us now comment on our constructive argument of Theorem \ref{thm-1}, namely, on how the parameters in the statement are properly determined. In fact, by taking into account the feature of our patch model as well as the normalization \eqref{TW-decay} of the traveling wave profiles at their leading edges, we consider the following ansatz for very negative times $t$:
\begin{equation}
\label{ansatz}
\begin{aligned}
u(t,x)=\begin{cases}
e^{-\lambda_1(x-c_1t)}~~\text{for}~x\le 0,\cr
e^{-\lambda_2(x-c_2t)}~~\text{for}~x\ge 0,
\end{cases}
\end{aligned}
\end{equation}
where $(c_1,\lambda_1)$ and $(c_2,\lambda_2)$ are chosen such that $c_1>c_1^*$, $c_2>c_2^*$, and
$$\left\{\baa{l}
\displaystyle0<\lambda_1=\lambda_1(c_1):=\frac{c_1-\sqrt{c_1^2-4d_1\mu_1}}{2d_1}<\lambda_1(c^*_1)=\sqrt{\frac{\mu_1}{d_1}},\vspace{3pt}\\
\displaystyle0<\lambda_2=\lambda_2(c_2):=\frac{c_2-\sqrt{c_2^2-4d_2\mu_2}}{2d_2}<\lambda_2(c^*_2)=\sqrt{\frac{\mu_2}{d_2}}.\eaa\right.$$
Due to the specific continuity and flux interface conditions at $x=0$, the ansatz \eqref{ansatz} leads to the following relations
$$\lambda_1c_1=\lambda_2c_2\ \hbox{ and }\ \lambda_1=\sigma \lambda_2.$$
Accordingly, we should have $d_1\lambda_1^2+\mu_1=\lambda_1c_1=\lambda_2c_2=d_2\lambda_2^2+\mu_2$, which, together with $\lambda_1=\sigma\lambda_2$, further yields
\begin{equation}\label{lambda_2}
0<\lambda_2=\sqrt{\frac{\mu_2-\mu_1}{d_1\sigma^2-d_2}}\ \hbox{ and }\lambda_2<\min\left(\sqrt{\frac{\mu_2}{d_2}},\sqrt{\frac{\mu_1}{d_1\sigma^2}}\right),
\end{equation}
and
\begin{align*}
\begin{cases}
\text{if}\ \displaystyle d_1\sigma^2>d_2,\ \hbox{ then }\ \mu_1<\mu_2< \mu_1\Big(2-\frac{d_2}{d_1\sigma^2}\Big),\vspace{3pt}\cr
\text{if}\ \displaystyle d_1\sigma^2<d_2,\ \hbox{ then }\ \mu_2<\mu_1< \mu_2\Big(2-\frac{d_1\sigma^2}{d_2}\Big).
\end{cases}
\end{align*}
This gives the condition~\eqref{mu1-mu2} on $d_i$, $\mu_i$ and $\sigma$.

Conversely, assuming~\eqref{mu1-mu2} and defining $\lambda_2$ and $\lambda_1$ as in~\eqref{lambda1-2}, we have~\eqref{lambda_2} and, in particular, $0<\lambda_i<\sqrt{\mu_i/d_i}$ for $i=1,2$, hence $c_i:=d_i\lambda_i+\mu_i/\lambda_i>2\sqrt{\mu_id_i}=c^*_i$. We also have $\lambda_1c_1=d_1\lambda_1^2+\mu_1=d_2\lambda_2^2+\mu_2=\lambda_2c_2$ by~\eqref{lambda1-2}, and then $c_1=\lambda_2c_2/\lambda_1=c_2/\sigma$, that is,~\eqref{c-12} holds.

The above heuristic arguments also explain why the condition $d_2\neq d_1\sigma^2$ is imposed. Indeed, if $d_2=d_1\sigma^2$, the above ansatz does not work, unless possibly in the particular case $\mu_1=\mu_2$.


\section{Proof of Theorem \ref{thm-1}}

The proof is divided into several steps: we first construct suitable super- and subsolutions of~\eqref{model} for very negative times. Then, by solving a sequence of Cauchy problems with initial times~$-n$ and  by passing to the limit as $n\to+\infty$, we obtain an entire solution~$u$ of~\eqref{model}, that is,~$u$ is defined for all $t\in\R$. Finally,  we show that this entire solution is truly a transition front with asymptotic past speed $c_1$ and asymptotic future speed $c_2$, based upon several auxiliary lemmas, where $c_1$ and $c_2$ are given in~\eqref{lambda1-2}-\eqref{c-12}. 
	

\subsection{Proper super- and subsolutions}\label{sec31}

Throughout the proof, we assume $d_2\neq d_1\sigma^2$ and~\eqref{mu1-mu2}. Let $(\lambda_2,c_2)$ and $(\lambda_1,c_1)$ be as in~\eqref{lambda1-2}-\eqref{c-12}. From the observations of the end of the previous section, we have
$$c_i>c^*_i=2\sqrt{d_i\mu_i}\ \hbox{(for $i=1,2$)},$$
and
\begin{equation}
\label{12}
c_1\lambda_1=c_2\lambda_2=d_2\lambda_2^2+\mu_2=d_1\lambda_1^2+\mu_1,~~\sigma c_1=c_2,\ \hbox{ and }\ \lambda_1<\sqrt{\frac{\mu_1}{d_1}},
\end{equation}
hence $\lambda_1$ is the smallest root of the equation $d_1\lambda^2-c_1\lambda+\mu_1=0$, that is,
\be\label{lambda1}
\lambda_1=\frac{c_1-\sqrt{c_1^2-4d_1\mu_1}}{2d_1}.
\ee
	
\subsubsection*{Construction of supersolutions}

For any $\gamma_1\ge c_1$, denoting
$$\Lambda_1:=\frac{\gamma_1-\sqrt{\gamma_1^2-4d_1\mu_1}}{2d_1},$$
we claim that the function $\overline u$ defined by
\begin{equation}
\label{super_}
\begin{aligned}
\overline u(t,x)=\begin{cases}
\min\left(V(x),e^{-\Lambda_1(x-\gamma_1 t)},e^{-\lambda_1(x-c_1 t)}\right),~~&x\le 0,\cr
e^{-\lambda_2 (x-c_2 t)},~~&x\ge 0,
\end{cases}
\end{aligned}
\end{equation}
is a generalized supersolution of \eqref{model}, for $t$ negative enough. Before proving the claim, we first note that, when $\gamma_1=c_1$, then $\Lambda_1=\lambda_1$ and $\overline u$ is reduced to the following:
\begin{equation}
\label{super}
\begin{aligned}
\overline u(t,x)=\begin{cases}
\min\left(V(x),e^{-\lambda_1(x-c_1 t)}\right),~~&x\le 0,\cr
e^{-\lambda_2 (x-c_2 t)},~~&x\ge 0.
\end{cases}
\end{aligned}
\end{equation}
\begin{figure}[H]
\centering
\includegraphics[scale=0.6]{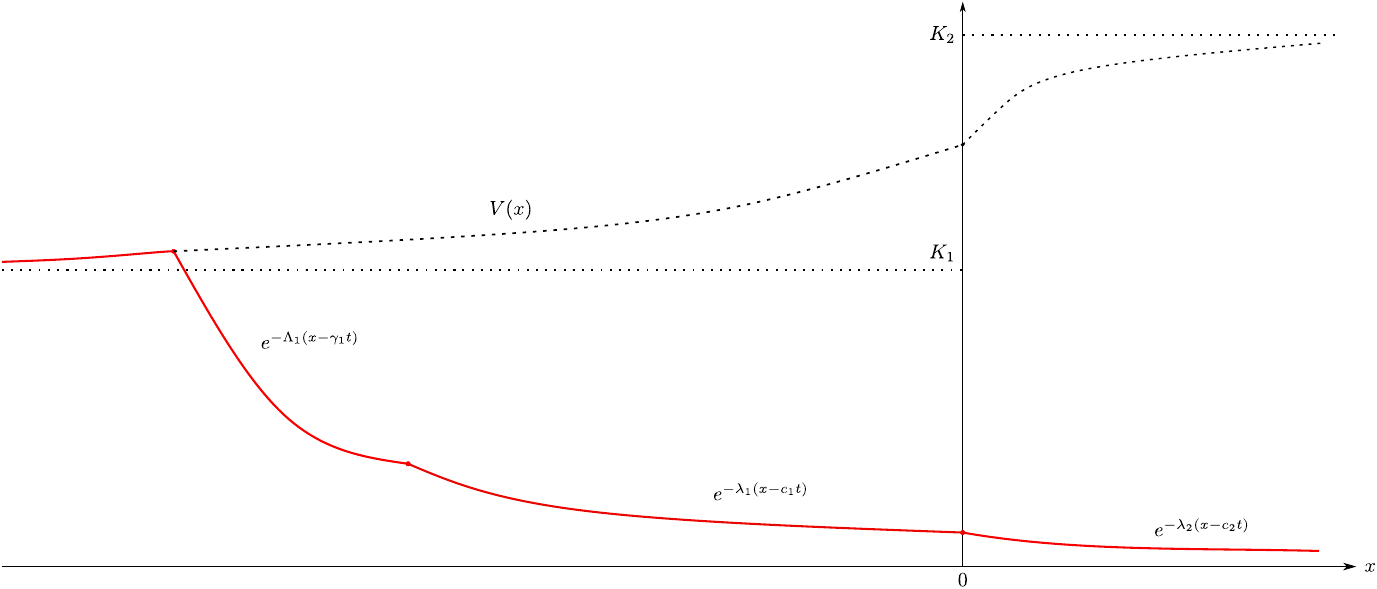}
\caption{Profile of the supersolution $\overline u$ with $\gamma_1\in[c_1,+\infty)$, in the case $K_1<K_2$. }
\end{figure}

To prove our claim, observe first that $0<\Lambda_1\le\lambda_1$, and even $0<\Lambda_1<\lambda_1$ if $\gamma_1>c_1$. By noticing that
$$0<\gamma_1\Lambda_1=d_1\Lambda_1^2+\mu_1\le d_1\lambda_1^2+\mu_1=c_1\lambda_1$$
(and even $\gamma_1\Lambda_1<c_1\lambda_1$ if $\gamma_1>c_1$), we have $e^{\lambda_1c_1t}\le  e^{\Lambda_1\gamma_1t}$ for all $t\le0$, whence we observe from~\eqref{super_} and the positivity of $\inf_{\R}V$ that there is $r>0$ such that $\overline{u}(t,x)=e^{-\lambda_1(x-c_1t)}$ for all $x\in(-r,0)$ and all $t$ negative enough. Moreover, since $ c_1\lambda_1=c_2\lambda_2$ by \eqref{12}, one readily verifies the continuity interface condition at $x=0$ for all $t$ negative enough. The flux condition at $x=0$ also holds since $\lambda_1=\sigma\lambda_2$, hence $-\lambda_1e^{\lambda_1c_1t}=-\sigma\lambda_2e^{\lambda_2c_2t}$ and $\overline{u}_x(t,0^-)=\sigma\overline{u}_x(t,0^+)$ for all $t$ negative enough. Eventually, it is easy to check that the functions $(t,x)\mapsto e^{-\Lambda_1(x-\gamma_1t)}$ and $(t,x)\mapsto e^{-\lambda_1(x-c_1t)}$ (resp. $(t,x)\mapsto e^{-\lambda_2(x-c_2t)}$) satisfy the equations of~\eqref{model} for all $t\in\R$ and $x<0$ (resp. $x>0$) with ``$=$'' replaced by ``$\ge$'', due to the KPP assumption~\eqref{hypkpp} on $f_i$ (for $i=1,2$), while $V$ is a stationary solution of~\eqref{model}. Therefore, we conclude that $\overline u$ is a generalized supersolution of~\eqref{model} for all $t$ negative enough and $x\in\R$. 

We also observe that, since $\inf_{\R}V>0$, one has
\be\label{overuV}
\overline{u}(t,x)\le V(x)
\ee
for all $t$ negative enough and for all $x\in\R$. 

\subsubsection*{Construction of subsolutions}

Let $\omega>0$ be given in \eqref{f-w}. Let us fix $\varep>0$ small enough such that
\begin{equation}
\label{varep}\left\{
\begin{aligned}
\lambda_1<\lambda_1+\sigma\varep<(1+\omega)\lambda_1,~~~~\lambda_2<\lambda_2+\varep<(1+\omega)\lambda_2,~~~~~~\\
\vartheta_1:=\sqrt{c_1^2-4d_1\mu_1}-d_1\sigma\varep>0,~~~\vartheta_2:=\sqrt{c_2^2-4d_2\mu_2}-d_2\varep>0.
\end{aligned}\right.
\end{equation}
Then, choose any $m$ such that
\begin{equation}
\label{value-m}
m>\max\Big(\frac{C}{\sigma\varep\vartheta_1}, \frac{C}{\varep\vartheta_2},1\Big)
\end{equation}
with $C$ given in \eqref{f-w}, and
\be\label{choicem}
\max\Big(\max_{x\in\R}\big(e^{-\lambda_1x}-m\,e^{-(\lambda_1+\sigma\varep)x}\big),\max_{x\in\R}\big(e^{-\lambda_2x}-m\,e^{-(\lambda_2+\varep)x}\big)\Big)<\min(K_1,K_2)=\inf_{\R}V.
\ee
	
We now claim that  there are $T>0$ and $x_0>0$ large enough such that, for any $\hat{c}_1\in(c^*_1,c_1]$, the function $\underline u$ defined by
\begin{equation}
\label{sub}
\begin{aligned}
\underline u(t,x)=\begin{cases}
\max\big(\hat\phi_1(x-\hat c_1 t+x_0),e^{-\lambda_1(x-c_1t)}-m\,e^{-(\lambda_1+\sigma\varep)(x-c_1t)}\big) ,~&x\le 0,\cr
\max\big(e^{-\lambda_2(x-c_2t)}-m\,e^{-(\lambda_2+\varep)(x-c_2t)},0\big),~~&x\ge 0,
\end{cases}
\end{aligned}
\end{equation}
is a generalized subsolution to \eqref{model} in $(-\infty,T]\times\R$, where $\hat\phi_1$ denotes the traveling front profile sol\-ving~\eqref{TW} and~\eqref{TW-decay} with $i=1$ and speed $\hat{c}_1$.
\begin{figure}[H]
\centering
\includegraphics[scale=0.6]{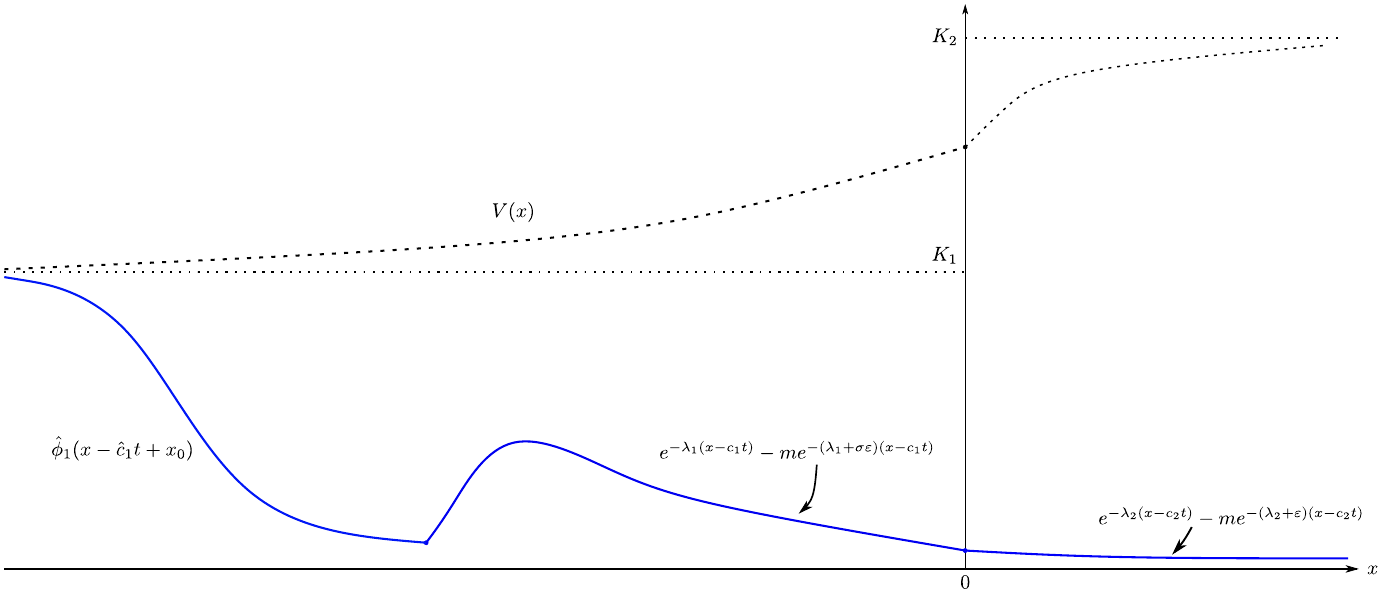}
\caption{Profile of the subsolution $\underline u$ with $\hat c_1\in(c_1^*, c_1]$, in the case $K_1<K_2$. }
\end{figure}

To prove the claim, observe first that, because $2\sqrt{d_1\mu_1}=c^*_1<\hat c_1\le c_1$, there holds
$$\hat\lambda_1:=\frac{\hat c_1-\sqrt{\hat c_1^2-4d_1\mu_1}}{2d_1}\ge\frac{c_1-\sqrt{c_1^2-4d_1\mu_1}}{2d_1}=
\lambda_1,$$
which further implies that $\hat c_1\hat\lambda_1=d_1\hat\lambda_1^2+\mu_1\ge d_1\lambda_1^2+\mu_1 =c_1\lambda_1$. One can then choose $T>0$ large enough such that $m<e^{\sigma\varep c_1 T}$ ($T$ can be chosen independently of $\hat{c}_1\in(c^*_1,c_1]$). Hence, for all $t\le-T$, there holds
$$e^{(\lambda_1c_1-\hat\lambda_1\hat c_1)t}\left(1-m\,e^{\sigma\varep c_1 t}\right)\ge e^{-(\lambda_1c_1-\hat\lambda_1\hat c_1)T}\left( 1-m\,e^{-\sigma\varep c_1 T}\right)\ge1-m\,e^{-\sigma\varep c_1T}=:\varsigma>0.$$
Next, there is $x_0>0$ sufficiently large such that $\varsigma>e^{-\lambda_1 x_0}$ ($x_0$ can be chosen independently of $\hat{c}_1\in(c^*_1,c_1]$). Thus, one has $\varsigma>e^{-\hat\lambda_1 x_0}$ and, for all $t\le -T$,
$$e^{\lambda_1c_1t}-m\,e^{(\lambda_1+\sigma\varep)c_1 t}=e^{\lambda_1c_1t}\left(1-m\,e^{\sigma\varep c_1 t}\right)\ge e^{\hat\lambda_1\hat c_1t}\varsigma> e^{\hat\lambda_1\hat c_1t-\hat\lambda_1 x_0}\ge\hat\phi_1(-\hat c_1 t+x_0)>0,$$
by~\eqref{ineqphii}. This implies that, for all $t\le -T$,
$$0<\underline u(t,0^-)=\max\left(\hat\phi_1(-\hat c_1 t+x_0), e^{\lambda_1c_1t}-m\,e^{(\lambda_1+\sigma\varep)c_1t}\right) =e^{\lambda_1c_1t}-m\,e^{(\lambda_1+\sigma\varep)c_1t},$$
and even that, by continuity, for every $T'<-T$, there is $s_1>0$ such that
$$\underline{u}(t,x)=e^{-\lambda_1(x-c_1t)}-m\,e^{-(\lambda_1+\sigma\varep)(x-c_1t)}\ \hbox{ for every $(t,x)\in[T',-T]\times(-s_1,0)$}.$$
Moreover,~\eqref{12} indicates that 
$$0<\underline u(t,0^-)= e^{\lambda_1c_1t}-m\,e^{(\lambda_1+\sigma\varep)c_1 t}=e^{\lambda_2c_2t}-m\,e^{(\lambda_2+\varep)c_2t}=\underline u(t,0^+)~~\text{for all}~t\le-T,$$
hence, by continuity, for every $T'<-T$, there is $s_2>0$ such that
$$\underline{u}(t,x)=e^{-\lambda_2(x-c_2t)}-m\,e^{-(\lambda_2+\varep)(x-c_2t)}\ \hbox{ for every $(t,x)\in[T',-T]\times(0,s_2)$}.$$
Therefore, $\underline u$ satisfies the continuity interface condition for $t\le -T$, and the flux interface condition at $x=0$ is satisfied as well, that is, $\underline u_x(t,0^-)=\sigma\underline u_x(t,0^+)$ for all $t\le-T$, due to~\eqref{lambda1-2}-\eqref{c-12},~\eqref{12}, and the previous observations.

It is left to check that the functions $(t,x)\mapsto\hat\phi_1(x-\hat c_1t+x_0)$ and
$$(t,x)\mapsto\underline{u}_{1,1}(t,x):=e^{-\lambda_1(x-c_1t)}-m\,e^{-(\lambda_1+\sigma\varep)(x-c_1t)}$$
(resp.
$$(t,x)\mapsto\underline{u}_{1,2}(t,x):=e^{-\lambda_2(x-c_2t)}-m\,e^{-(\lambda_2+\varep)(x-c_2t)})$$ satisfy~\eqref{model} for all $t\le-T$ and $x<0$ when equal to $\underline{u}(t,x)$ (resp. for all $t\le-T$ and $x>0$ such that $\underline{u}_{1,2}(t,x)\ge0$, that is, $\underline{u}(t,x)=\underline{u}_{1,2}(t,x)$), with ``$=$'' replaced by ``$\le$''. Let us first consider $I_1=(-\infty,0)$. The function $(t,x)\mapsto\hat\phi_1(x-\hat c_1t+x_0)$ satisfies the equation $u_t=d_1 u_{xx}+f_1(u)$ in $\R\times I_1$. Let us then consider the set of points $(t,x)\in(-\infty,-T]\times(-\infty,0)$ where $\underline u(t,x)=\underline{u}_{1,1}(t,x)=e^{-\lambda_1(x-c_1t)}-m\,e^{-(\lambda_1+\sigma\varep)(x-c_1t)}\ge\hat\phi_1(x-\hat c_1t+x_0)>0$. For such~$(t,x)$, one has $x-c_1 t>0$ (since otherwise $\underline{u}(t,x)$ would be nonpositive, as $m>1$), and then $\underline{u}_{1,1}(t,x)<e^{-\lambda_1(x-c_1t)}<1$. Then, from~\eqref{f-w},~\eqref{12}-\eqref{lambda1} and~\eqref{varep}-\eqref{value-m}, one derives that
\begin{align*}
(\underline{u}_{1,1})_t(t,x)-d_1(\underline{u}_{1,1})_{xx}(t,x) &=\mu_1\underline{u}_{1,1}(t,x)-m\sigma\varep\vartheta_1 e^{-(\lambda_1+\sigma\varep)(x-c_1t)}\cr
&=\mu_1\underline{u}_{1,1}(t,x)-m\sigma\varep\vartheta_1 [\underbrace{e^{-\lambda_1(x-c_1t)}}_{\ge\underline{u}_{1,1}(t,x)}]^{1+\omega}\underbrace{e^{-[(\lambda_1+\sigma\varep)-(1+\omega)\lambda_1](x-c_1t)}}_{\ge1}\cr
&\le \mu_1\underline{u}_{1,1}(t,x)- m\sigma\varep\vartheta_1(\underline{u}_{1,1}(t,x))^{1+\omega}\cr
&\le f_1(\underline{u}_{1,1}(t,x)).
\end{align*}
Similarly, as $x-c_2t>0$ and $\underline{u}(t,x)<e^{-\lambda_2(x-c_2t)}<1$ for all $t\le-T<0$ and $x>0$, a straightforward computation, for any $(t,x)\in(-\infty,-T]\times(0,+\infty)$ such that $0\le\underline{u}(t,x)=e^{-\lambda_2(x-c_2t)}-m\,e^{-(\lambda_2+\varep)(x-c_2t)}=\underline{u}_{1,2}(t,x)$, yields
$$\baa{rcl}
(\underline{u}_{1,2})_t(t,x)-d_2(\underline{u}_{1,2})_{xx}(t,x) & = & \mu_2\underline{u}_{1,2}(t,x)-m\varep\vartheta_2 e^{-(\lambda_2+\varep)(x-c_2t)}\vspace{3pt}\\
& = & \mu_2\underline{u}_{1,2}(t,x)-m\varep\vartheta_2 [e^{-\lambda_2(x-c_2t)}]^{1+\omega} e^{-[(\lambda_2+\varep)-(1+\omega)\lambda_2](x-c_2t)}\vspace{3pt}\\
& \le & \mu_2\underline{u}_{1,2}(t,x)-m\varep\vartheta_2(\underline{u}_{1,2}(t,x))^{1+\omega}\vspace{3pt}\\
& \le & f_2(\underline{u}_{1,2}(t,x)).\eaa$$
One then concludes that $\underline u$ is a generalized subsolution of \eqref{model} in $[T',-T]\times\R$ for every $T'<-T$.

\subsubsection*{Conclusion}

We consider in particular the case $\gamma_1=\hat c_1=c_1$. Combining the constructions of $\overline u$ (which requires that $\gamma_1\ge c_1$) and $\underline u$ (which requires that $c^*_1<\hat c_1\le c_1$), problem \eqref{model} admits a generalized supersolution $\overline u$ given by~\eqref{super}, as well as a generalized subsolution $\underline u$ given by~\eqref{sub} (with $\hat\phi_1=\phi_1$ here) in $(-\infty,-T]\times\R$ for some large enough $T>0$ and $x_0>0$, so that all above inequalities hold, including~\eqref{overuV} in $(-\infty,-T]\times\R$ (even if it means increasing $T>0$). 
\begin{figure}[H]
\centering
\includegraphics[scale=0.6]{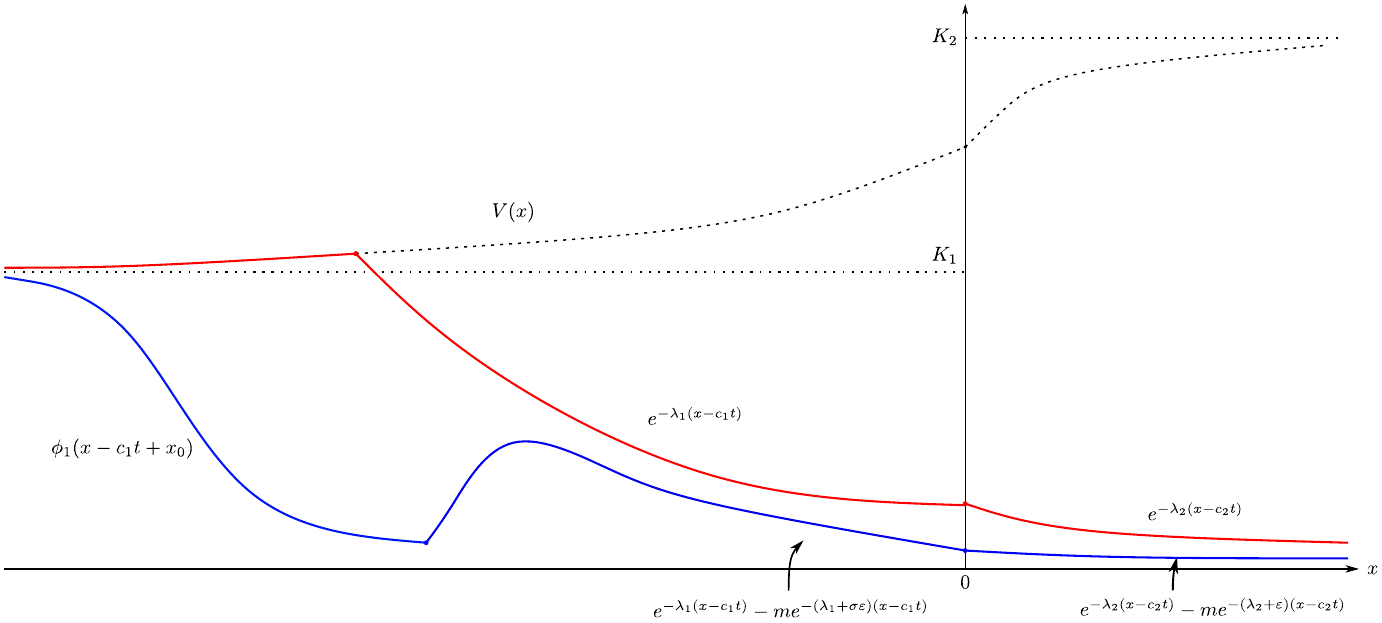}
\caption{Profiles of a coexisting pair of supersolution $\overline u$ (in red) and subsolution $\underline u$ (in blue), in the case $K_1<K_2$.}
\end{figure}
Moreover, it is known from~\cite{AW2} that
$$K_1-\phi_1(x)\sim\alpha\,e^{a x}~\text{as}~x\to-\infty,~~\text{with}~~a=\frac{-c_1+\sqrt{c_1^2-4d_1f'_1(K_1)}}{2d_1}>0,$$
for some $\alpha>0$. However, 
$$K_1-V(x)=O(e^{bx})~~\text{as}~x\to-\infty,~~\text{with}~~b=\sqrt{\frac{-f'_1(K_1)}{d_1}}>0.$$
Since $0<a<b$, we see that $V(x)$ converges faster to $K_1$ as $x\to-\infty$ than $\phi_1(x)$. Remember also that $\phi_1(+\infty)=0$, $\phi_1$ is decreasing, and $\inf_{\R}V>0$. Then, for all $t\le -T$ (up to increasing~$T$ if needed), it follows that $V(x)>\phi_1(x-c_1t+x_0)$ for all $x\le0$, whatever the relation of $K_1$ and $K_2$ is. Together with~\eqref{super},~\eqref{choicem}-\eqref{sub} and
$$\phi_1(x-c_1t+x_0)\le e^{-\lambda_1(x-c_1t+x_0)}\le e^{-\lambda_1(x-c_1t)}$$
(by~\eqref{ineqphii} and the positivity of $x_0$), it follows that
$$0\le\underline u(t,x)\le \overline u(t,x)~~\text{for all}~(t,x)\in(-\infty,-T]\times\mathbb{R}.$$
	 

\subsection{Construction of an entire solution $u$}\label{sec32}

A standard limiting argument now gives an entire solution to \eqref{model}. Indeed, for each $n\in\mathbb{N}$ with~$n>T$, let $u_n$ be the solution of the Cauchy problem associated with~\eqref{model} in $[-n,+\infty)\times\R$ with initial (at time $-n$) datum defined by
$$u_n(-n,x)= \underline u(-n,x)~~\hbox{for all} ~x\in\mathbb{R}.$$
The comparison principle stated in Proposition~\ref{prop-cp} and its following extension, applied in $[-n,-T]\times\R$, gives that
$$\max(K_1,K_2)\ge\overline u(t,x)\ge u_n(t,x)\ge \underline u(t,x)\ \hbox{ for all $(t,x)\in[-n,-T]\times\R$}.$$
Furthermore, $\max(K_1,K_2)\ge u_n(t,x)\ge0$ for all $(t,x)\in[-n,+\infty)\times\R$, from Proposition~\ref{prop-cp} applied this time in $[-n,T']\times\R$ for every $T'>-n$. In particular, one has that
$$u_n(-n+1,x)\ge\underline u(-n+1,x)=u_{n-1}(-n+1,x)\ \hbox{ for all $n\in\N$ with $n>T+1$ and $x\in\R$}.$$
It follows from the comparison principle again that $u_n(t,x)\ge u_{n-1}(t,x)$ for every $n>T+1$ and every $(t,x)\in[-n+1,+\infty)\times\R$. Therefore, for each $(t,x)\in\R\times\R$, the sequence $(u_n(t,x))_{n\in\N,\,n>\max(T,-t)+1}$ is nondecreasing and bounded. From the Schauder estimates of Proposition~\ref{prop-wellposedness+Schauder}, the functions $u_n$ converge as $n\to+\infty$, locally uniformly in $(t,x)\in\R\times\R$, to a classical bounded entire solution $u$ of \eqref{model}. Moreover,
\be\label{underoveru}
0\le\underline u (t,x)\le u(t,x) \le \overline u(t,x) ~~\text{in}~(-\infty,-T]\times\mathbb{R}.
\ee

Lastly, since $\underline{u}>0$ in $(-\infty,-T]\times(-\infty,0]$, one has $u>0$ in $(-\infty,-T]\times(-\infty,0]$. The strong parabolic maximum principle applied to the nonnegative function $u$ in $(-\infty,-T]\times[0,+\infty)$ then yields $u>0$ in $(-\infty,-T]\times[0,+\infty)$, hence $u>0$ in $(-\infty,-T]\times\R$. Finally,
$$u(t,x)>0\ \hbox{ for all }(t,x)\in\R\times\R$$
from the strong parabolic maximum principle again and Hopf lemma (at $x=0)$, or from Proposition~\ref{prop-cp}.
  

\subsection{The entire solution $u$ is a transition front of~\eqref{model}}

More precisely, we will show that~\eqref{X(t)} holds with
\begin{align}\label{defXt}
X(t)=
\begin{cases}
c_1 t~~~\text{if}~t\le -T,\cr
c_2 t~~~\text{if}~t>-T.
\end{cases}
\end{align}
We point out that bounded perturbations of $X(t)$ would not affect~\eqref{X(t)}. Therefore,~\eqref{X(t)} would also hold if $X$ in~\eqref{defXt} is replaced by $\tilde{X}:\R\to\R$ defined by $\tilde{X}(t)=c_1t$ for $t\le0$ and $\tilde{X}(t)=c_2t$ for $t>0$.

For $t\le -T$, we observe from the construction of super- and subsolutions above and from $V(-\infty)=K_1$, that
$$\begin{cases}
\overline u(t,x+c_1 t)\to K_1~~~\text{as}~x\to-\infty,~\text{uniformly in}~t\le -T,\cr
\underline u(t,x+c_1 t)\to K_1~~~\text{as}~x\to-\infty,~\text{uniformly in}~t\le-T,\end{cases}$$
and
$$\begin{cases}
\overline u(t,x+c_1 t)\to 0~~~\text{as}~x\to+\infty,~\text{uniformly in}~t\le -T,\cr
\underline u(t,x+c_1 t)\to 0~~~\text{as}~x\to+\infty,~\text{uniformly in}~t\le-T.\end{cases}$$
It then follows from~\eqref{underoveru} and $V(-\infty)=K_1$ again that
\be\label{transition1}\begin{cases}
u(t,x)-V(x)\to 0 &\text{as}~x-c_1t\to-\infty,~\text{uniformly in}~t\le -T,\cr
u(t,x)\to 0 &\text{as}~x-c_1t\to+\infty,~\text{uniformly in}~t\le -T.
\end{cases}
\ee

To show that $u$ is a transition front of~\eqref{model} in the sense of Definition~\ref{def} with $X$ given by~\eqref{defXt}, it is left to discuss the case that $t\ge -T$ and show that
\begin{align}
\label{X(t)=c_2t}
\begin{cases}
u(t,x)-V(x)\to 0 &\text{as}~x-c_2t\to-\infty,~~\text{uniformly in}~t\ge -T,\\
u(t,x)\to 0 &\text{as}~x-c_2t\to+\infty,~~\text{uniformly in}~t\ge -T.	
\end{cases}
\end{align}
 
For this purpose, we shall make use of some auxiliary lemmas. We begin with proving the exponential decay of~$u$ far ahead of the moving interface $x=c_2t$.

\begin{lemma}\label{lemma0}
There holds that
\begin{equation}
\label{asymptotics of u}
u(t,x)\sim e^{-\lambda_2(x-c_2t)}~~~\text{as}~x-c_2t\to+\infty,~\text{uniformly in}~t\ge-T.
\end{equation}	
\end{lemma}

\begin{proof}
Let $(\lambda_1,c_1)$ and $(\lambda_2,c_2)$ be given in~\eqref{lambda1-2}-\eqref{c-12}. We borrow the idea from the construction of $\overline u$ in~\eqref{super} and define $\overline v$  as follows:
\begin{align*}
\overline v(t,x)=\begin{cases}
e^{-\lambda_1(x-c_1 t)},&  x\le 0,\cr
e^{-\lambda_2 (x-c_2 t)},&x\ge 0.
\end{cases}
\end{align*}
We observe that $\overline u(t,x)\le \overline v(t,x)$ for $(t,x)\in[-T,+\infty)\times\R$, and, as in Section~\ref{sec31}, it is easily checked that $\overline v$ is a generalized supersolution of \eqref{model} for $(t,x)\in[-T,+\infty)\times\R$ (and even in~$\R\times\R$). Moreover, there holds $u(-T,\cdot)\le \overline u(-T,\cdot)\le \overline v(-T,\cdot)$ in $\R$, thanks to \eqref{underoveru}. The comparison principle in Proposition~\ref{prop-cp} implies that
\begin{equation}
\label{u-1}
u(t,x)\le \overline v(t,x)~~\hbox{for}~(t,x)\in[-T,+\infty)\times\R.
\end{equation}

On the other hand, let $\omega>0$, $\varep>0$ and $m>0$ be given in \eqref{f-w}, \eqref{varep} and \eqref{value-m}-\eqref{choicem}, respectively. Choose $M>0$ large enough such that 
\begin{equation}
\label{M}
M>\max(e^{\varep c_2T},m)>1.
\end{equation} 
Let us now introduce the function $\underline{v}$ defined in $[-T,+\infty)\times\R$ by
\begin{align}\label{underv}
\underline v(t,x)=&\displaystyle\begin{cases}
0, &x\le 0,\\
\max\big(e^{-\lambda_2(x-c_2t)}-M e^{-(\lambda_2+\varep)(x-c_2 t)},0\big), &x\ge 0.\end{cases}
\end{align}
We aim to show that $\underline v$ is a generalized subsolution of \eqref{model} in $[-T,+\infty)\times\R$. Indeed, since $M>e^{\varep c_2T}$, one has that
$$e^{\lambda_2c_2t}-Me^{(\lambda_2+\varep)c_2t}=e^{\lambda_2c_2t}(1-Me^{\varep c_2t})< e^{\lambda_2c_2t}(1-e^{\varep c_2(t+T)})\le 0~~\text{for all}~t\ge -T,$$
which implies that $\underline v(t,\cdot)=0$ in the vicinity of the origin for each $t\ge -T$. Furthermore, since the profiles of $\underline{v}(t,\cdot)$ are shifted to the right with speed $c_2>0$ as time $t$ runs, one can find $s>0$ such that $\underline v=0$ in $[-T,+\infty)\times(-\infty,s]$. One then deduces that $\underline v$ automatically satisfies the continuity and flux interface conditions at $x=0$ for every $t\ge -T$. Following a similar computation as for $\underline u$ and utilizing the choice of $\omega$, $\varep$ and $M>m$ in~\eqref{f-w},~\eqref{varep}-\eqref{choicem} and~\eqref{M}, one then gets that the function $(t,x)\mapsto e^{-\lambda_2(x-c_2t)}-M e^{-(\lambda_2+\varep)(x-c_2 t)}$ satisfies the second equation of \eqref{model} in $I_2=(0,+\infty)$ with ``='' replaced by ``$\le$'' for those $(t,x)$ in $[-T,+\infty)\times(0,+\infty)$ such that $e^{-\lambda_2(x-c_2t)}-M e^{-(\lambda_2+\varep)(x-c_2 t)}\ge0$. Moreover, by noticing that $M>m$ and remembering~\eqref{underoveru}, we observe that
$$0\le\underline v(-T,\cdot)\le \underline u(-T,\cdot)\le u(-T,\cdot)~~~\text{in}~\R.$$
Consequently, $\underline v$ is a generalized subsolution of \eqref{model} in $[-T,+\infty)\times\R$. In particular, one infers from comparison principle that
\begin{equation}
\label{u-2}
\underline v(t,x)\le u(t,x)~~\hbox{for all}~(t,x)\in[-T,+\infty)\times\R.
\end{equation}	
Combining \eqref{u-1} and \eqref{u-2}, along with the structures of $\overline v$ and $\underline v$, one reaches the desired conclusion \eqref{asymptotics of u}, which completes the proof of Lemma~\ref{lemma0}.
\end{proof}

Next, we aim to show the large time convergence of $u$ to the positive stationary solution~$V$ far behind the moving interface $x=c_2t$ for $t\ge-T$. To do so, we prove the following lemma as preparation.
	 
\begin{lemma}
\label{lemma1}
For any fixed $\bar x\in\R$, there holds that
$$u(t,x)- V(x)\to 0~~~\text{as}~t\to+\infty,~\hbox{uniformly in}~x\le \bar x.$$
\end{lemma}
 
\begin{proof}
Notice first from~\eqref{overuV} and~\eqref{underoveru} that $0\le u(t,x)\le\overline{u}(t,x)\le V(x)$ in $(-\infty,-T]\!\times\!\R$, which, together with the comparison principle stated in Proposition \ref{prop-cp} yields $u(t,x)\le V(x)$ for all $(t,x)\in[-T,+\infty)\times\R$, hence
\be\label{inequV}
0\le u(t,x)\le V(x)\ \hbox{ for all }(t,x)\in\R\times\R.
\ee
The main idea now, among other things, is to construct a suitable subsolution such that the entire solution $u$ can be forced to converge to $V$ far to the left for large times. 

To do so, we take $R>0$ large enough  such that
$$\frac{\pi}{2R}<\sqrt{\frac{f'_1(0)}{d_1}}.$$
Define $\Psi:\R\to\R$ as
$$\begin{aligned}
\Psi(x)=\begin{cases}
\ 1&\text{in} ~(-\infty,-R),\cr
\ \displaystyle\cos\Big(\frac{\pi}{2R}(x+R)\Big)&\text{in}~[-R,0],\cr
\ 0 &\text{in}~(0,+\infty).
\end{cases}
\end{aligned}$$
The function $\Psi$ is continuous in $\R$, $C^1$ in $\R\setminus\{0\}$, and $C^2$ in $\R\setminus\{-R,0\}$. Due to the choice of $R$, there exists $\eta_0\in(0,K_1)$ small enough such that $\eta d_1\Psi''(x)+f_1(\eta\Psi(x))\ge0$ for all $x\in\R\setminus\{-R,0\}$ and for all $\eta\in(0,\eta_0)$. Fix now~$x_1\in(-\infty,-R)$ sufficiently negative and $\eta\in(0,\eta_0)$ such that
$$\eta\Psi(\cdot-x_1)<\underline u(-T,\cdot)\le u(-T,\cdot)\ \hbox{ in $(-\infty,0]$},$$
which is possible thanks to~\eqref{sub}-\eqref{underoveru}, and $\hat\phi_1(\xi)\to K_1$ as $\xi\to-\infty$ (actually, here, $\hat\phi_1=\phi_1$ since $\hat{c}_1=c_1$). Denote by $w$ the solution of the following initial-boundary value problem:
\begin{equation}
\label{w-Dirichlet}
\begin{aligned}
\begin{cases}
w_t=d_1 w_{xx}+f_1(w)&\text{for} ~t>0,~x<0,\cr
w(t,0)=0	&\text{for} ~t\ge 0,\cr
w(0,x)=\eta\Psi(x-x_1) &\text{for} ~x\le 0.
\end{cases}
\end{aligned}
\end{equation}
The strong parabolic maximum principle entails that $w(t,x)>w(0,x)=\eta\Psi(x-x_1)\ge0$ for all $t>0$ and $x<0$, whence $w(t+h,\cdot)>w(t,\cdot)$ in $(-\infty,0)$ for every $h>0$ and $t\ge 0$. That is,~$w$ is increasing with respect to $t\ge0$ in the space interval $(-\infty,0)$. On the other hand, since
$$w(0,\cdot)=\eta\Psi(\cdot-x_1)<\underline{u}(-T,\cdot)\le\overline{u}(-T,\cdot)\le V\ \hbox{ in $(-\infty,0]$},$$
we readily verify that the positive stationary solution $V$ of \eqref{model} is a  supersolution of \eqref{w-Dirichlet} and the strong maximum principle and the Hopf lemma at $x=0$ give that $w(t,x)<V(x)$ for all $t\ge0$ and $x\in(-\infty,0]$.  From parabolic estimates, it follows that  $w(t,\cdot)$ converges as $t\to+\infty$ in $C^2_{loc}((-\infty,0])$, to a positive bounded stationary solution $p\in C^2((-\infty,0])$  of~\eqref{w-Dirichlet}. The function $p$ satisfies $p(0)=0$ and
$$\eta\Psi(x-x_1)<p(x)\le V(x)\ \hbox{ for all $x\in(-\infty,0)$}.$$
Moreover, we claim that
\begin{equation}
\label{p}
p(-\infty)=K_1.
\end{equation}
To prove this, consider an arbitrary sequence $(x_n)_{n\in\mathbb{N}}$ in $(-\infty,0]$ diverging to  $-\infty$ as $n\to+\infty$ and define $p_n:=p(\cdot+x_n)$ in $(-\infty,0]$ for each $n\in\mathbb{N}$. Then, by standard elliptic estimates, the sequence $(p_n)_{n\in\mathbb{N}}$ converges as $n\to+\infty$, up to extraction of some subsequence, in $C^2_{loc}(\mathbb{R})$ to a bounded function $p_\infty$ which solves
$$d_1p_\infty''+f_1(p_\infty)=0\ \hbox{ in $\mathbb{R}$}.$$
Moreover, $\inf_{\R}p_\infty\ge\eta>0$. It follows that~$p_\infty\equiv K_1$ in $\R$, due to the hypothesis that $f_1>0$ in $(0,K_1)$ and $f_1<0$ in $(K_1,+\infty)$. That is,~$p_n\to K_1$ as~$n\to+\infty$ in $C^2_{loc}(\R)$. Since the  sequence~$(x_n)_{n\in\mathbb{N}}$ was arbitrarily chosen, it follows that  $p(x)\to K_1$ and $p'(x)\to0$ as $x\to-\infty$. Thus,~\eqref{p} is achieved. 

Since $w(0,x)=\eta\Psi(x-x_1)<\underline{u}(-T,x)\le u(-T,x)$ for all $x\le0$ and $w(t,0)=0<u(t-T,0)$ for all $t\ge 0$, we deduce from the comparison principle that
$$w(t,x)<u(t-T,x)\ \hbox{ for all $t\ge 0$ and $x\in(-\infty,0]$}.$$
Together with~\eqref{inequV}, passing to the limit as $t\to+\infty$ gives
\begin{equation}
\label{2.12}
p\le \liminf_{t\to+\infty}u(t,\cdot)\le \limsup_{t\to+\infty}u(t,\cdot)\le V~~~\text{locally uniformly in}~ (-\infty,0].
\end{equation}

Consider now any $\delta\in(0,K_1/2)$. Since $p(-\infty)=V(-\infty)=K_1$, one can choose $x_2\in(-\infty,x_1-R]\subset(-\infty,0)$ negative enough in such a way that
\be\label{x2delta}
K_1-\delta\le p(x)\le  V(x)\le K_1+\delta\ \hbox{ for all $x\le x_2$}.
\ee
Then, thanks to~\eqref{inequV} and~\eqref{2.12}, one derives the existence of a sufficiently large $T^*>0$  such that
\be\label{K1delta}
K_1-2\delta\le p(x_2)-\delta\le u(t,x_2)\le V(x_2)\le K_1+\delta~~\text{for all}~t\ge T^*.
\ee
Moreover, since $x_2\le x_1-R$, we also notice that
$$\inf_{x\le x_2}u(T^*,x)\ge\inf_{x\le x_2}w(T+T^*,x)\ge\inf_{x\le x_2}w(0,x)=\inf_{x\le x_2}\eta\Psi(x-x_1)=\eta>0.$$
Consider now the solution of the ODE $\zeta'(t)=f_1(\zeta(t))$ for $t\ge T^*$ associated with the initial condition $\zeta(T^*)=\min(\eta,K_1-2\delta)\in(0,K_1-2\delta]$. One has $\zeta(t)\nearrow K_1$ as $t\to+\infty$ by~\eqref{hypkpp}, and there is a unique $\overline{T}\in[T^*,+\infty)$ such that $\zeta(\overline{T})=K_1-2\delta$. Using $\zeta$ as a subsolution to~\eqref{model} for~$t\in[T^*,\overline{T}]$ and $x\le x_2$, the comparison principle asserts that $\zeta(t)\le u(t,x)$ for all $t\in[T^*,\overline{T}]$ and $x\le x_2$. In particular, $u(\overline{T},x)\ge K_1-2\delta$ for all $x\le x_2$, and then
$$u(t,x)\ge K_1-2\delta\ \hbox{ for all $t\ge\overline{T}$ and $x\le x_2$}$$
from~\eqref{K1delta} and the maximum principle (since $f_1(K_1-2\delta)>0$).
Together with~\eqref{inequV} and~\eqref{x2delta}, one gets that
\be\label{uV}
\limsup_{t\to+\infty}\Big(\sup_{x\le x_2}|u(t,x)-V(x)|\Big)\le3\delta.
\ee

On the other hand, it follows from~\cite[Theorem 2.6]{HLZ2} on the large-time behavior of~\eqref{model} in the KPP-KPP frame that each function $u_n$, as defined in Section~\ref{sec32}, has the property $u_n(t,x)\to V(x)$ as $t\to+\infty$ locally uniformly in $x\in\R$, whence so does $u$ by~\eqref{inequV} and the inequality $u_n\le u$ in $[-n,+\infty)\times\R$. Consequently,~\eqref{uV} holds with $x_2$ replaced by any $\bar x\in\R$. Lastly, since $\delta>0$ can be arbitrarily small, the proof of Lemma~\ref{lemma1} is thereby complete.
\end{proof}

\begin{lemma}
\label{lemma2}
There holds that
$$\limsup_{t\to+\infty}\Big(\sup_{x\le c_2t+A}|u(t,x)-V(x)|\Big)\to0\ \hbox{ as }A\to-\infty.$$
\end{lemma}

\begin{proof}
First of all, remember that $M$, in the definition~\eqref{underv} of $\underline v$ in Lemma~\ref{lemma0}, satisfies \eqref{M}, that $\varep>0$ is such that \eqref{varep} holds,  and that $\lambda_2$ is given by \eqref{lambda1-2}. One can then pick $A_1>0$  large enough such that
\be\label{defX0}
0<\max(e^{-\lambda_2 A_1}, Me^{-\varep A_1})<\min(1, K_2).
\ee

Consider now any $\delta\in(0,K_2/3)$. Since $V(+\infty)=K_2$, there exists $B>0$ sufficiently large such that
\be\label{defB}
K_2-\delta\le V(x)\le K_2+\delta~~\hbox{for all }x\ge B.
\ee
Since  $u(t,x)\to V(x)$ as $t\to+\infty$ locally uniformly in $\R$ (again by~\cite[Theorem~2.6]{HLZ2}) and since $u(t,x)\le V(x)$ for all $(t,x)\in\R\times\R$ by~\eqref{inequV}, there is $T_1>0$ large enough such that
\begin{equation}
\label{2.14}
c_2T_1>B,\ \hbox{ and }\ K_2-2\delta\le u(t,B)\le V(B)\le K_2+\delta~~\hbox{for all }t\ge T_1.
\end{equation}
The inequalities~\eqref{inequV} and~\eqref{defB} also entail that
\be\label{BT2}
u(t,x)\le V(x)\le K_2+\delta\ \hbox{ for all $t\ge T_1$ and $x\ge B$}.
\ee
Furthermore, by using the subsolution $\underline v$ given by~\eqref{underv} in Lemma~\ref{lemma0}, one infers from~\eqref{u-2} and~\eqref{defX0} that, for every $t\ge T_1$,
\be\label{u-lower boundary2}\baa{rcl}
u(t,c_2 t+A_1)\ge\underline v(t, c_2t+A_1) & = & \max\big(e^{-\lambda_2A_1}-Me^{-(\lambda_2+\varep)A_1},0\big)\vspace{3pt}\\
& = & e^{-\lambda_2A_1}\big(1-M e^{-\varep A_1}\big)=:\gamma\in(0,\min(1,K_2)).\eaa
\end{equation}
Since $u$ is continuous and positive in $\R\times\R$, there is $\rho>0$ such that
$$u(T_1,x)\ge\rho\ \hbox{ for all $x\in[B,c_2T_1+A_1]$}$$
(remember that $c_2T_1>B$ and $A_1>0$). Denote
$$\kappa:=\min(K_2-3\delta,\gamma,\rho)\in(0,K_2-3\delta].$$
Since $f_2(\kappa)>0$, it follows from the maximum principle together with~\eqref{2.14} and~\eqref{u-lower boundary2} that
\be\label{BX0}
u(t,x)\ge\kappa>0\ \hbox{ for all $t\ge T_1$ and $B\le x\le c_2t+A_1$}.
\ee

We finally claim that there is $A_2<0$ such that
\be\label{claimX1}
\liminf_{t\to+\infty}\Big(\min_{B\le x\le c_2t+A_2}u(t,x)\Big)\ge K_2-3\delta.
\ee
Indeed, otherwise, by~\eqref{BT2}-\eqref{u-lower boundary2}, there would exist $L\in[\kappa,K_2-3\delta]$ and some sequences $(t_n)_{n\in\N}$ in $\R$ and $(x_n)_{n\in\N}$ in $[B,+\infty)$ such that $t_n\to+\infty$, $x_n-c_2t_n\to-\infty$ and $u(t_n,x_n)\to L$ as $n\to+\infty$. Up to extraction of a subsequence, two cases may occur: either $x_n\to+\infty$ as~$n\to+\infty$, or there is $\tilde{B}\in[B,+\infty)$ such that $x_n\to\tilde{B}$ as $n\to+\infty$. In the former case, from standard parabolic estimates together with~\eqref{BT2} and~\eqref{BX0}, the functions
$$U_n:(t,x)\mapsto U_n(t,x):=u(t+t_n,x+x_n)$$
converge, up to extraction of another subsequence, locally uniformly in $\R^2$, to a classical solution~$U_\infty$ of $(U_{\infty})_t=d_2(U_{\infty})_{xx}+f_2(U_\infty)$ in $\R^2$, such that
$$\kappa\le U_\infty\le K_2+\delta\ \hbox{ in $\R^2$}$$
and~$U_\infty(0,0)=L\in[\kappa,K_2-3\delta]$. Let $\underline{\zeta}$ and $\overline{\zeta}$ be the solutions of the ODEs $\underline{\zeta}'(t)=f_2(\underline{\zeta}(t))$ and $\overline{\zeta}'(t)=f_2(\overline{\zeta}(t))$ for $t\ge0$, with initial conditions
$$\underline{\zeta}(0)=\kappa\in(0,K_2-3\delta]\ \hbox{ and }\ \overline{\zeta}(0)=K_2+\delta.$$
It follows from the maximum principle that $\underline{\zeta}(t-t')\le U_\infty(t,x)\le\overline{\zeta}(t-t')$ for all $t'\le t\in\R$ and $x\in\R$. Since $\underline{\zeta}(+\infty)=\overline{\zeta}(+\infty)=K_2$ by~\eqref{hypkpp}, one infers that $U_\infty\equiv K_2$ in $\R^2$, a contradiction with $U_\infty(0,0)=L\le K_2-3\delta<K_2$.

In the case where $x_n\to\tilde{B}\in[B,+\infty)$ as $n\to+\infty$, the functions
$$\tilde{U}_n:(t,x)\mapsto\tilde{U}_n(t,x):=u(t+t_n,x)$$
converge, up to extraction of another subsequence, locally uniformly in $\R\times(0,+\infty)$, to a classical solution $\tilde{U}_\infty$ of $(\tilde{U}_{\infty})_t=d_2(\tilde{U}_{\infty})_{xx}+f_2(\tilde{U}_\infty)$ in $\R\times(0,+\infty)$, such that
$$\kappa\le\tilde{U}_\infty\le K_2+\delta\ \hbox{ in $\R\times[B,+\infty)$}$$
and $\tilde{U}_\infty(0,\tilde{B})=L\in[\kappa,K_2-3\delta]$. Furthermore, by~\eqref{2.14}, one has $\tilde{U}_\infty(t,B)\ge K_2-2\delta$ for all $t\in\R$. For any $t'\in\R$, since $\tilde{U}_\infty(t',\cdot)\ge\kappa$ in $[B,+\infty)$, the maximum principle then implies that $\tilde{U}_\infty(t,\cdot)\ge\underline{\zeta}(t-t')$ for all $t\in[t',t'+\tau]$, where $\tau>0$ is the unique time such that $\underline{\zeta}(\tau)=K_2-2\delta$. Therefore, $\tilde{U}_\infty\ge K_2-2\delta$ in $\R\times[B,+\infty)$, contradicting $\tilde{U}_\infty(0,0)=L\le K_2-3\delta<K_2-2\delta$.

As a conclusion, the claim~\eqref{claimX1} has been proved. Together with Lemma~\ref{lemma1},~\eqref{inequV},~\eqref{defB} and~\eqref{BT2}, one gets that
$$\limsup_{t\to+\infty}\Big(\sup_{x\le c_2t+A_2}|u(t,x)-V(x)|\Big)\le4\delta.$$
Since $\delta>0$ can be arbitrarily small, the proof of Lemma~\ref{lemma2} is thereby complete.
\end{proof}

With the preliminary above lemmas in hand, we can finally show that $u$ is a transition front connecting $V$ and $0$ in the sense of Definition~\ref{def}, with $X$ defined in~\eqref{defXt}. We recall that only~\eqref{X(t)=c_2t} remains to be proved. Let $\delta>0$ be arbitrary. As an immediate consequence of Lemmas~\ref{lemma0} and~\ref{lemma2}, one infers the existence of $X_2>0$ and $T_2>0$ such that
\begin{align}\label{cdn-2}
\begin{cases}
|u(t,x)-V(x)|\le\delta &\text{for all $t\ge T_2$ and $x-c_2t\le-X_2$},\\
0<u(t,x)\le\delta &\text{for all $t\ge-T$ and $x-c_2t\ge X_2$},	
\end{cases}
\end{align}
In view of $u(-T,-\infty)=K_1$ due to~\eqref{underoveru} and the definitions~\eqref{super} and~\eqref{sub} of the functions~$\overline{u}$ and~$\underline{u}$, it follows from parabolic estimates and $f_1(K_1)=0$ that $u(t,-\infty)=K_1$ uniformly for $t\in[-T,T_2]$. We also notice that $V(-\infty)=K_1$, whence
\begin{equation}
\label{cdn-3}
u(t,x)-V(x)\to 0~~\text{as}~x-c_2t\to -\infty,~~\text{uniformly in}~t\in[-T,T_2].
\end{equation}
Gathering \eqref{cdn-2}-\eqref{cdn-3} together with the arbitrariness of $\delta>0$,~\eqref{X(t)=c_2t} follows. Finally, $u$ is a transition front of~\eqref{model} connecting $V$ and $0$ in the sense of Definition~\ref{def}, and~\eqref{X(t)} holds with~$X$ as in~\eqref{defXt}.


\subsection{End of the proof of Theorem~\ref{thm-1}}

It remains to show the limit properties~\eqref{lim+infty0}-\eqref{lim+infty}. Assume first, by way of contradiction, that~\eqref{lim+infty0} does not hold. Then there are $\delta>0$, and some sequences $(t_n)_{n\in\N}$ and $(x_n)_{n\in\N}$ in $\R$ such that $t_n\to-\infty$ as $n\to+\infty$, and
\be\label{ineqdelta}
|u(t_n,x_n)-\phi_1(x_n-c_1t_n)|\ge\delta>0\ \hbox{ for all }n\in\N.
\ee
From~\eqref{transition1} together with $\phi_1(-\infty)=V(-\infty)=K_1$ and $\phi_1(+\infty)=0$, it follows that the sequence $(x_n-c_1t_n)_{n\in\N}$ is bounded, hence converges to a real number $\xi$, up to extraction of a subsequence. From standard parabolic estimates, the functions
$$V_n:(t,x)\mapsto V_n(t,x):=u(t+t_n,x+c_1t_n)$$
converge locally uniformly in $\R^2$, up to extraction of another subsequence, to a classical solution~$V_\infty$ of $(V_{\infty})_t=d_1(V_\infty)_{xx}+f_1(V_\infty)$ in $\R^2$. From~\eqref{inequV} together with $V(-\infty)=K_1$, one gets that
\be\label{Vinfty}
0\le V_\infty\le K_1\ \hbox{ in $\R^2$}.
\ee
Furthermore, the definitions~\eqref{super} and~\eqref{sub} of $\overline{u}$ and $\underline{u}$, together with~\eqref{underoveru} again, imply that
\be\label{decay}
\max\big(\phi_1(x-c_1t+x_0),e^{-\lambda_1(x-c_1t)}-m\,e^{(\lambda_1+\sigma\varep)(x_1-c_1t)}\big)\le V_\infty(t,x)\le e^{-\lambda_1(x-c_1t)}
\ee
for all $(t,x)\in\R^2$, with $\varep>0$ as in~\eqref{varep}, whereas the limit $\phi_1(-\infty)=K_1$ and the inequalities~\eqref{Vinfty}-\eqref{decay} imply that $V_\infty(t,x)\to K_1$ as $x-c_1t\to-\infty$, uniformly in $t\in\R$. Together with~\eqref{TW-decay} applied to $\phi_1$ and $c_1>c^*_1$, we can now adapt the Liouville-type result given in \cite{HNRR2013} and we infer that
\be\label{Vphi}
V_\infty(t,x)=\phi_1(x-c_1t)\ \hbox{ for all $(t,x)\in\R^2$}.
\ee
More precisely, that conclusion follows from \cite[Theorem~3.5]{BH2007}, which is based on the above precise exponential decay~\eqref{decay} and on the sliding method. On the other hand, the inequa\-lity~\eqref{ineqdelta} and the convergence $\lim_{n\to+\infty}x_n-c_1t_n=\xi\in\R$ yield $|V_\infty(0,\xi)-\phi_1(\xi)|\ge\delta>0$, contradicting~\eqref{Vphi}. As a consequence,~\eqref{lim+infty0} has been shown.

Let us finally prove~\eqref{lim+infty}. Similarly as in the previous paragraph, we assume, by way of contradiction, that~\eqref{lim+infty} does not hold. Then there are $\tilde\delta>0$, and some sequences $(\tilde t_n)_{n\in\N}$ and $(\tilde x_n)_{n\in\N}$ in $\R$ such that $\lim_{n\to+\infty}\tilde t_n=\lim_{n\to+\infty}\tilde x_n=+\infty$ and
\be\label{ineqdelta2}
|u(\tilde t_n,\tilde x_n)-\phi_2(\tilde x_n-c_2\tilde t_n)|\ge\tilde\delta>0\ \hbox{ for all }n\in\N.
\ee
From~\eqref{X(t)=c_2t} together with $\phi_2(-\infty)=V(+\infty)=K_2$ and $\phi_2(+\infty)=0$, it follows that the sequence $(\tilde x_n-c_2\tilde t_n)_{n\in\N}$ is bounded, hence converges to a real number $\tilde\xi$, up to extraction of a subsequence. From standard parabolic estimates, the functions
$$\tilde V_n:(t,x)\mapsto\tilde{V}_n(t,x):=u(t+\tilde t_n,x+c_2\tilde t_n)$$
converge locally uniformly in $\R^2$, up to extraction of another subsequence, to a classical solution~$\tilde V_\infty$ of $(\tilde V_{\infty})_t=d_2(\tilde V_\infty)_{xx}+f_2(\tilde V_\infty)$ in $\R^2$. From~\eqref{inequV} together with $V(+\infty)=K_2$, one gets that $0\le \tilde V_\infty\le K_2$ in $\R^2$. Furthermore, Lemma~\ref{lemma0} imples that
$$\tilde V_\infty(t,x)\sim e^{-\lambda_2(x-c_2t)}\ \hbox{ as $x-c_2t\to+\infty$, uniformly in $t\in\R$},$$
whereas Lemma~\ref{lemma2} and the limit $V(-\infty)=K_2$ imply that $\tilde V_\infty(t,x)\to K_2$ as $x-c_2t\to-\infty$, uniformly in $t\in\R$. Together with~\eqref{TW-decay} applied to $\phi_2$ and $c_2>c^*_2$, it follows as in the above paragraph that $\tilde V_\infty(t,x)=\phi_2(x-c_2t)$ for all $(t,x)\in\R^2$. This is in contradiction with the inequality $|\tilde V_\infty(0,\tilde\xi)-\phi_2(\tilde\xi)|\ge\tilde\delta>0$ derived from~\eqref{ineqdelta2} and $\lim_{n\to+\infty}\tilde x_n-c_2\tilde t_n=\tilde\xi\in\R$. As a consequence,~\eqref{lim+infty} has been shown, and the proof of Theorem~\ref{thm-1} is thereby complete.\hfill$\Box$


\end{document}